\documentclass[11pt,a4paper]{article}
\usepackage[centertags]{amsmath}
\usepackage[colorlinks]{hyperref}
\usepackage{amsfonts}
\usepackage{graphics}
\usepackage{graphicx}
\usepackage{amssymb}
\usepackage{amsthm}
\usepackage{newlfont}
\usepackage{epsfig}
\usepackage[final]{pdfpages}
\usepackage{xcolor}
\usepackage{marginnote}
\usepackage{latexsym}
\usepackage{color}
\usepackage{amssymb}
\usepackage{amsmath}
\usepackage{amsthm} \usepackage{graphicx}
\newtheorem{ill}{Illustration}

\usepackage[top=2.5cm,bottom=2.5cm,left=2.5cm,right=2.5cm]{geometry}
\usepackage{cite,amsmath,amssymb}
    \usepackage[margin=1cm,%
                font=small,%
                format=hang,%
                labelsep=period,%
                labelfont=bf]{caption}
                
\newtheoremstyle{theorem}
  {10pt}          
  {10pt}  
  {\sl}  
  {\parindent}     
  {\bf}  
  {. }    
  { }    
  {}     
\theoremstyle{theorem}
\newtheorem{theorem}{Theorem}
\newtheorem{corollary}[theorem]{Corollary}

\newtheorem{problem}{Problem}

\newtheoremstyle{defi}
  {8pt}          
  {8pt}  
  {\rm}  
  {\parindent}     
  {\bf}  
  {. }    
  { }    
  {}     
\theoremstyle{defi}

\date{}
\begin{document}
\title{Steiner Wiener index of block graphs}
\maketitle
\vspace{-20mm}

\begin{center}
\author{Matja\v{z} Kov\v{s}e,   
  \and
    Rasila V A,
    \and
    	Ambat Vijayakumar
 }
	\end{center}
\newcommand{\Addresses}{
  \bigskip
  \footnotesize

  Matja\v{z} Kov\v{s}e, School of Basic Sciences, 
IIT Bhubaneswar,
India, \texttt{matjaz.kovse@gmail.com}
\par\nopagebreak
  Rasila V A,
  Department of Mathematics, Cochin University of Science and Technology, India, \texttt{17rasila17@gmail.com} \par\nopagebreak  
  Ambat Vijayakumar,
	 Department of Mathematics, Cochin University of Science and Technology, India, \texttt{vambat@gmail.com}}

\begin{abstract}
Let $S$ be a set of vertices of a connected graph $G$. The Steiner distance of $S$ is the minimum size of a connected subgraph of $G$ containing all the vertices of $S$. The Steiner $k$-Wiener index is the sum of all Steiner distances on sets of $k$ vertices of $G$. Different simple methods for calculating the Steiner $k$-Wiener index of block graphs are presented. 
\end{abstract}
\noindent
{\bf MR Subject Classifications:} 05C12
\bigskip\noindent


\section{Introduction}
All graphs in this paper are simple, finite and undirected. Unless stated otherwise let $n=|V(G)|$ and $m=|E(G)|$, hence $n$ denotes the order and $m$ size of a graph $G$. 
If $G$ is a connected graph
and $u, v \in V (G)$, then the (geodetic) distance $d_G(u, v)$ (or simply $d(u, v)$ if there is no confusion about $G$) between $u$ and $v$ is the number of edges on a shortest path connecting $u$ and $v$. 
The Wiener index $W (G)$ of a connected graph $G$ is defined as
$$W (G) =\sum_ {\{u,v \} \in V(G)} d(u,v).$$
The first investigation of this distance-based graph invariant was done by Wiener in 1947, who realized in \cite{wiener1947} that there exist correlations between the
boiling points of paraffins and their molecular structure and noted that in the case of a tree it can be easily calculated
from the edge contributions by the following formula:
\begin{equation}\label{eq:Wiener}
 W(T)= \sum_{e \in E(T)}n(T_{1})n(T_{2}),
\end{equation}
where $n(T_{1} )$ and $n(T_{2} )$ denote the number of vertices in connected components $T_{1}$ and $T_{2}$ formed by removing an edge $e$ from the tree $T$ .

The Steiner distance of a graph has been introduced  in \cite{chartrand1989} by Chartrand et al., as a natural generalization of the geodetic graph
distance. For a connected graph $G$ and $S\subseteq
V(G)$, the Steiner distance $d_G(S)$ (or simply $d(S)$) among the vertices of
$S$ is the minimum size among all connected subgraphs whose vertex
sets contain $S$. Note that any such subgraph $H$ is a tree, called a Steiner tree connecting vertices from $S$. Vertices of $S$ are called terminal vertices of tree $H$, while the rest of the vertices of $H$ are called inner vertices of the Steiner tree $H$. If $S = \{u, v\}$, then $d(S) = d(u, v)$ coincides with the geodetic
distance between $u$ and $v$. In \cite{dankelmann1996average} Dankelmann et al. followed by studying
the average $k$-Steiner distance $\mu_k (G)$. In \cite{gutman2016discuss}, Li et al. introduced a generalization of the Wiener index by using the Steiner distance. The Steiner $k$-Wiener index $SW_{k}(G)$ of a connected
graph $G$ is defined as $$SW_k(G) =  \sum_ {\substack{S\subseteq V(G)\\ | S| =k}} d(S).$$
For $k = 2$, the Steiner $k$-Wiener index coincides with the Wiener index. The average $k$-Steiner distance $\mu_k (G)$ is related to the Steiner $k$-Wiener index via the equality $\mu_k (G) = SW_{k}(G)/ {n \choose k}$.
In \cite{gutman2016discuss} the exact values of the Steiner $k$-Wiener index of
the path, star, complete graph, and complete bipartite graph and sharp lower and upper bounds for $SW_{k}(G)$ for connected graphs and for trees have been obtained. In \cite{gutman2015multicenter} an application of Steiner $k$-Wiener index in mathematical chemistry is reported, and it is shown that the term $W(G)+ \lambda SW_k(G)$ provides a better  approximation for the boiling points of alkanes than $W(G)$ itself, and that the best such approximation is obtained for $k=7$. 
For a survey on Steiner distance see \cite{mao2017survey}.

The problem of deciding whether for a given subset of vertices in a graph $G$ there exist a Steiner tree of size at most $t$ belongs to the classical NP-complete problems from \cite{garey1979NP}. Hence for $k>2$ it is not very likely to find an efficient way to compute the Steiner $k$-Wiener index for general graphs. Therefore it becomes interesting to either find efficient procedures to compute the Steiner $k$-Wiener index or bounds for particular classes of graphs.

A vertex $v$ is a cut vertex of graph $G$ if deleting $v$ and all edges incident to it increases
the number of connected components $G$. A block of a graph is a maximal connected vertex induced
subgraph that has no cut vertices. A block graph is a graph in which every block is a clique. Block graphs are a natural generalization of trees, and they arise in areas such as metric graph theory, \cite{bandelt1986block}, 
molecular graphs \cite{behtoei2010block} 
and phylogenetics \cite{dress2016block}. They have been characterized in various ways, for example, as certain intersection graphs \cite{harary1963block}, or in terms of distance conditions \cite{behtoei2010block}. In \cite{bapat2011block} it has been shown that  the determinant of the distance matrix of a block graph depends only on types of blocks, and not how they are connected. Steiner distance, Steiner centers and Steiner medians of block graphs have been studied in \cite{yeh2008centers}.

A vertex of a block graph $G$ that appears in only one block is called a pendant vertex. Hence there are exacly two types of vertices in a block graph: cut vertices and pendant vertices. A block $B$ is called a pendant block if $B$ has a non-empty intersection with a unique block in $G$. Any block graph different from a complete graph has at least two pendant blocks. Pendant vertex from a pendant block of $G$ is called a leaf of $G$. 
For a block graph $G$ with blocks $B_1,B_2, \ldots, B_t$, let $b_i=|V(B_i)|$, for $i\in \{1, \ldots, t\}$. We call a sequence $b_1 \geq b_2 \geq  \ldots \geq b_t$ the block order sequence of $G$.

A line graph $L(G)$ of a simple graph $G$ is obtained by associating a vertex with each edge of the graph and connecting two vertices with an edge if and only if the corresponding edges of $G$ have a vertex in common. A graph is a line graph of a tree if it is a connected block graph in which each cut vertex is in exactly two blocks, hence they are claw-free block graphs: no induced subgraph is a claw - a complete bipartite graph $K_{1,3}.$ A caterpillar is a tree with the property that a path remains
if all leaves are deleted. This path is called the backbone of the caterpillar.
 Line graphs of caterpillars are called path-like block graphs. These are all block graphs with precisely two pendant blocks. 
The windmill graph $Wd(r,t)$ is a block graph constructed by joining $t$ copies of $K_r$ at a shared vertex, where $r, t \geq 2$. If $v \in V(G)$ is adjacent to all other vertices of $G$, it is called a universal vertex. A block graph with a universal vertex is called star-like block graph.

For $v\in V(G)$ let $N(v)$ denote the set of all neighbours. The degree of vertex $v\in V(G)$ in graph $G$ is defined as the number of neighbours and denoted with $deg(v)$, i.e., $deg(v)=|N(v)|$. Graph $G$ has degree sequence $\lambda = (d_1, \ldots, d_n)$, $d_1 \geq d_2 \geq \ldots \geq d_k$,  if vertices of $G$ can be indexed
from $v_1$ to $v_n$ such that $deg(v_i)=d_i$, $i=1,\ldots, n$. Let
$\mathcal{T}(\lambda)$ be the set of trees with degree sequence $\lambda$.
It is known that $\mathcal{T}(\lambda)$ is not empty, if and only if $d_1+\ldots+d_n=2(n-1)$.

The Cartesian product $G \,\square\, H$ of two graphs $G$ and $H$ is the graph with vertex set $V(G)\times V(H)$ and
$(a,x)(b,y)\in E(G\,\square\, H)$ whenever either $ab\in E(G)$ and $x=y$, or $a=b$ and $xy\in E(H)$.
Cartesian products of complete graphs are called Hamming graphs. They can be alternatively described
as follows. For $i=1,2, \ldots, t$ let $r_i\geq 2$ be given integers. Let $G$ be the graph whose vertices are the $t$-tuples $a_1,a_2,\ldots a_t$ with $a_i\in \{0,1, \ldots, r_i-1\}$. 
Two vertices being
adjacent if the corresponding tuples differ in precisely 
one place. Then it is straightforward to see that $G$ is
isomorphic to $K_{r_1}\Box K_{r_2}\Box \cdots \Box K_{r_t}$. A subgraph $H$ of $G$ is called isometric (or distance preserving) if  $d_{H}(u,v)=d_{G}(u,v)$ for all $u,v\in V(H)$. Isometric subgraphs of Hamming graphs are called partial Hamming graphs. 
The Hamming distance between two $t$-tuples is defined as the number of positions in which these $t$-tuples differ, in a partial Hamming graph it coincides with the geodetic distance.

In this paper we obtain several simple methods for calculating the Steiner $k$-Wiener index of block graphs. In Section \ref{section:blockdecomposition} we present the block decomposition formula of the Steiner $k$-Wiener index of block graphs. 
In Section \ref{section:edgedecomposition} we present the edge decomposition formula of the Steiner $3$-Wiener index of block graphs.  In Section \ref{section:vertexdecomposition} we present the vertex decomposition formula of the Steiner $k$-Wiener index of block graphs and relate it to the $k$-Steiner betweenness centrality. 
In Section \ref{section:extremes} we study the graphs which minimize or maximize Steiner $k$-Wiener index among all block graphs with the same set of blocks, and obtain the sharp lower bound. To describe block graphs that maximize the Steiner $k$-Wiener index we introduce a special graph transformation called generalized block shift, which generalizes the generalized tree shift transformation introduced in \cite{csikvari2010poset} by Csikv{\'a}ri. 

\section{Block decomposition formula of Steiner $k$-Wiener index of block graphs}
\label{section:blockdecomposition}

Let $n(G)$ denote the number of vertices of a graph $G$. For a graph $G$ with $p$, $p >1$, connected components $G_1, G_2, \ldots, G_p$ we denote by $N_k(G)$ the sum over all partitions of $k$ into at least two nonzero parts of products of combinations distributed among the $p$ components of $F$:
$$
N_k(G)=\sum_{\substack{
l_1+l_2+\ldots+l_p=k\\
0\leq l_1,l_2,\ldots,l_p<k}} {n(G_1)\choose l_1} {n(G_2)\choose l_2} \ldots {n(G_p)\choose l_p} 
$$

For a graph $G$ and $e\in E(G)$, let $G - e$ denote a graph obtained by removing
edge $e$ from $G$. Then the following formula for a tree $T$ has been shown in \cite{gutman2015multicenter, kovse2016, gutman2016discuss} $$ SW_{k}(T)=\sum_{e \in E(T)}N_{k}(T -e).$$

For a given partition $l_1+l_2+\ldots+l_p=k$, let $\alpha(l_1,l_2,\ldots,l_p)$ denote the number of nonzero summands minus 1.
For a graph $G$ with $p$, $p >1$, connected components $G_1, G_2, \ldots, G_p$, we define $N_k'(G)$ to be the sum over all partitions of $k$ into at least two nonzero parts of products of combinations distributed among the $p$ components of $G$ multiplied by $\alpha(l_1,l_2,\ldots,l_p):$ 
$$
N_k'(G)=\sum_{\substack{
l_1+l_2+\ldots+l_p=k\\
0\leq l_1,l_2,\ldots,l_p<k}} {n(G_1)\choose l_1} {n(G_2)\choose l_2} \ldots {n(G_p)\choose l_p}  \cdot \alpha(l_1,l_2,\ldots,l_p).
$$
For a connected graph $G$, we define $N_k'(G)=0$. Note that by the definition ${n\choose 0}=1$, and ${n\choose k}=0$ whenever $n<k$.

Let $G \setminus B_i$ denote graph obtained from $G$ by deleting all edges from block $B_i$.

\begin{theorem}\label{thm1}
Let $G$ be a connected block graph with blocks $B_1, B_2,\ldots, B_t$.
Then $$SW_k(G) =  \sum_{i=1}^t N_k'(G \setminus B_i).$$
\end{theorem}
\begin{proof}   
$N_k'(G \setminus B_i)$ counts the contribution of $B_{i}$ to the Steiner $k$-Wiener index of vertices from $G \setminus B_i$. Let $G_{1},G_{2},\ldots G_{p}$ be the connected components of $G\setminus B_i$. For a given partition $l_1+l_2+\ldots+l_p=k$, let $\alpha(l_1,l_2,\ldots,l_p)$ denote the number of nonzero summands minus one. A block $B_{i}$ of $G$ contributes $\alpha(l_1,l_2,\ldots,l_p)$ to the Steiner distance of $k$ vertices from $G \setminus B_i$.  Then,
$$
N_k'(G \setminus B_{i})=\sum_{\substack{
l_1+l_2+\ldots+l_p=k\\
0\leq l_1,l_2,\ldots,l_p<k}} {n(G_1)\choose l_1} {n(G_2)\choose l_2} \ldots {n(G_p)\choose l_p}  \cdot \alpha(l_1,l_2,\ldots,l_p).
$$  
and the formula follows.
\end{proof}

The block decomposition of a block graph $G$, with blocks $B_1, B_2,\ldots, B_t$ of orders $b_1, b_2,\ldots, b_t$, allows us also to obtain isometric embedding into the Hamming graph $K_{b_1} \Box K_{b_2} \Box \ldots \Box K_{b_t}$ as follows. For $i \in \{1, \ldots, t \}$, we assign arbitrarily different values from $\{0,1, \ldots, b_i -1\}$ to the $i$-th coordinates of vertices from $B_i$. All vertices from the same connected component $C$ of $G \setminus B_i$ get the same value for the $i$-th coordinate as the already labeled vertex from $C \cap B_i$. Now we can show that the corresponding Hamming labelling of vertices of $G$ can be further used to calculate the Steiner distance on a subset of more than two vertices.

\begin{theorem}\label{thm1}
Let $G$ be a connected block graph on $n$ vertices with blocks $B_1, B_2,\ldots, B_t$, of orders $b_1 \leq b_2 \leq \ldots \leq b_t$. Let $S \subseteq V(G)$, $|S| = k$, $2\leq k \leq n$.

Then $$d(S) =  \sum_{i=1}^t \ell_i(S) - t,$$
where $\ell_i(S)$ denotes the number of different values assigned to the $i$-th coordinates of vertices from $S$.
\end{theorem}
\begin{proof}   
If all vertices from $S$ belong to the same connected component of $G \setminus B_i$ then the contribution of the block $B_i$ to the Steiner distance between vertices from $S$ is 0 in all other cases it is equal to $\ell_i(S) -1$, and hence the formula follows.
\end{proof}

\begin{ill}
Consider the block graph $G$ from the Figure \ref{figblock}. Let $\ell(v)$ denote the Hamming labelling induced by an isometric embedding into $K_{2} \Box K_{3} \Box K_{4}$. Then one possible labeling of vertices of $G$ reads as follows:
$\ell(v_1)=(0,0,0)$, $\ell(v_2)=(0,1,0)$, $\ell(v_3)=(0,2,0)$, $\ell(v_4)=(1,2,0)$, $\ell(v_5)=(1,2,1)$, $\ell(v_6)=(1,2,2)$ and $\ell(v_7)=(1,2,3)$. Thus for $S=\{v_1, v_2, v_5, v_6\}$, we have $\ell_1(S)=2, \ell_2(S)=\ell_3(S) = 3$, and hence $d(S)=2+3 +3 -3= 5$.

\begin{center}
\begin{figure}[h]
\centering
  \includegraphics[width= 60mm]{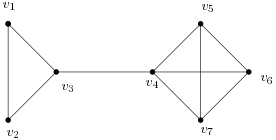}
  \caption{A block graph $G$ with vertices $v_1, v_2, v_3, v_4, v_5, v_6, v_7$ and blocks of orders $b_1=2,b_2=3$ and $b_3=4$.}
  \label{figblock}
\end{figure}
\end{center}
\end{ill}

\section{Edge decomposition of the Steiner $k$-Wiener index of block graphs}
\label{section:edgedecomposition}
For an edge $ab \in E(G)$ we denote $W_{ab} = \{ v\in V(G)\ | \ d_G(a,v) < d_G(b,v) \},$ $ W_{ba} = \{ v\in V(G)\ | \ d_G(b,v) < d_G(a,v) \}$, and
$_{a}W_{b} = \{ v\in V(G)\ | \ d_G(a,v) =d_G(b,v) \}$. Hence for any edge $ab \in E(G),$ vertices of $G$ are partitioned into three sets: $W_{ab}$ - a set of vertices that are closer to $a$ than to $b$, $W_{ba}$ - a set of vertices that are closer to $b$ than to $a$ and $_{a}W_{b}$ - a set of vertices that are at the same distance to $a$ and $b$. If $G$ is a bipartite graph, then $_{a}W_{b} = \emptyset$. If $ab$ is a cut edge then also $_{a}W_{b} = \emptyset$. Additionaly we denote $N_{ab}= |W_{ab}|, N_{ba}=|W_{ba}|$ and $_aN_{b}=|_{a}W_{b}|$. Note that for a pendant vertex $a$ and its neighbour $b$ in a block graph $G$, $N_{ab}=1$, and $_aN_{b} = |V(G)|-2$. For a cut vertex $a$ and its neighbour $b$ in a block graph $G$, $N_{ab}>1$.

\begin{theorem}\label{thm.Wedge} Let $G$ be a connected block graph on $n$ vertices. Then
\begin{equation*}
\displaystyle W(G) = \sum_{\substack{ab \in E(G)}}N_{ab} \cdot N_{ba}.
\end{equation*}
\end{theorem}

\begin{proof}
Let $u,v \in V(G)$. Note that any pair of vertices in a block graph is connected by a unique shortest path. Edge $ab$ appears on a shortest path between $u$ and $v$ if and only if $ W_{ab}$ contains one of the vertices $u,v$ and $W_{ba}$ the other vertex. The expression on the right side of the formula thus counts the number of times edge $ab$ appears as an edge in a shortest path between a pair of vertices $u$ and $v$, hence contributing 1 to $d(u,v)$.
\end{proof}

\begin{theorem}\label{thm.SW3edge} Let $G$ be a connected block graph on $n$ vertices. Then
\begin{equation*}
SW_3(G) = \sum_{\substack{ab \in E(G)}}N_{ab} \cdot N_{ba} + \frac 23 \sum_{ab \in E(G)} N_{ab} \cdot N_{ba} \cdot {_aN_{b}}.
 \end{equation*}
\end{theorem}

\begin{proof}
Let $\{u,v,w\}=S \subseteq V(G)$. Edge $ab$ appears in some Steiner tree connecting vertices of $S$ if and only if $S \cap W_{ab} \neq \emptyset$ and $S \cap W_{ba} \neq \emptyset$. The expression on the right side counts the contribution of an edge $ab$ to Steiner 3-Wiener index of $G$. We distinguish two cases.\\

{\bf Case 1.} $S \cap W_{ab} \neq \emptyset$, $S \cap W_{ba} \neq \emptyset$ and $S \cap _aW_{b} = \emptyset$.\\
In this case $ab$ appears in any Steiner tree connecting vertices of $S$, therefore it contributes 1 to $d(S)$, altogether this happens $N_{ab} \cdot N_{ba}$ times. \\

{\bf Case 2.} $S \cap W_{ab} \neq \emptyset$, $S \cap W_{ba} \neq \emptyset$ and $S \cap _aW_{b} \neq \emptyset$.\\
In this case each of $W_{ab}, W_{ba}$ and $_aW_{b}$ includes one of the vertices from $S$. W. l. o. g. let $u \in W_{ab}$, $v \in W_{ba}$ and $w \in _aW_{b}$.
Let $c$ be a common neighbour of $a$ and $b$ that lies on a shortest path between $a$ and $w$ (as well on a shortest path between $b$ and $w$). Any Steiner tree on $S$ must include $a,b$ and $c$. Note that $d(\{a,b,c\})=2$ and there are three different Steiner trees connecting $a,b$ and $c$, each of them induced by a pair of edges among $ab, ac$ and $bc$. Therefore the contribution of each edge among $ab, ac$ and $bc$ to the Steiner distance on $S$ is $\frac 23$. There are exactly $N_{ab} \cdot N_{ba} \cdot {_aN_{b}}$ number of $\{u,v,w\}$ sets with the property that any Steiner tree connecting $\{u,v,w\}$ includes $a$ and $b$ and some common neighbour $c$ of $a$ and $b$. Since the contribution to the Steiner distance $d(\{u,v,w\})$ of each edge among $ab,ac$ and $bc$ is defined to be $\frac 23$, the common contribution of  $ab,ac$ and $bc$  in the sum  $\frac 23 \sum_{ab \in E(G)} N_{ab} \cdot N_{ba} \cdot {_aN_{b}}$ equals $\frac 23 + \frac 23 +\frac 23 = 2$, which is exactly the contribution of $\{a,b,c\}$ to the Steiner distance $d(\{u,v,w\})$.
\end{proof}

\begin{theorem}\label{thm.SW3edge2}
Let $G$ be a connected block graph on $n$ vertices. Then

\begin{equation*}
\displaystyle SW_3(G) = \left(\sum_{\substack{ab \in E(G)}} \left[ {n \choose 3} - {n - N_{ab} \choose 3} - {n - N_{ba} \choose 3} - {{_aN_{b}}\choose 3} \right] \right) - \frac 13 \sum_{ab \in E(G)} N_{ab} \cdot N_{ba} \cdot {_aN_{b}}.
\end{equation*}
\end{theorem}

\begin{proof} For $ab \in E(G)$ let $n_3(ab)$ denote the number of subsets $S \subseteq V(G)$, $|S|=3$, for which $ab$ appears as an edge in a Steiner tree connecting vertices of $S$. Note that an edge $ab$ will not be contained in any Steiner tree connecting vertices of $S$ if and only if one of the cases holds:
case 1: $S \subseteq W_{ab}$, case 2: $S \subseteq W_{ba}$, case 3: $S \subseteq \, _{a}W_{b}$, case 4: $S \cap \, W_{ab} \neq \emptyset$, $S \cap \, _{a}W_{b} \neq \emptyset$ and $S \cap \,W_{ba} = \emptyset$, case 5: $S \cap \, W_{ba} \neq \emptyset$, $S \cap \, _{a}W_{b} \neq \emptyset$ and $S \cap \,W_{ab} = \emptyset$. Therefore it follows
\begin{align*}
n_3(ab)&= {n \choose 3} - {N_{ab}\choose 3} - {N_{ba}\choose 3} - { _aN_{b}\choose 3} 
-\left[{N_{ab} +  \,_aN_{b}\choose 3} - {N_{ab}\choose 3} - {_aN_{b}\choose 3} \right] \\ &-\left[ {N_{ba} + \, _aN_{b}\choose 3} - {N_{ba}\choose 3} - {_aN_{b}\choose 3} \right]
={n \choose 3} - {n - N_{ba} \choose 3} - {n - N_{ab} \choose 3} -  {{_aN_{b}}\choose 3}.
\end{align*}
 

Note that $\displaystyle{\sum_{ab \in E(G)} n_3(ab)}$ has to be adjusted because of the double counting of the common edge contributions of $ab, ac$ and $bc$ to the Steiner distances of three vertices $u,v$ and $w$ as in the Case 2 in the proof of Theorem \ref{thm.SW3edge}. For an edge $ab$, the number of such sets equals $N_{ab} \cdot N_{ba} \cdot {_aN_{b}}$.  The common contribution of $ab, ac$ and $bc$ to the Steiner 3-Wiener index in $\sum_{ab \in E(G)} n_3(ab)$ is three instead of two. In $\sum_{ab \in E(G)} N_{ab} \cdot N_{ba} \cdot {_aN_{b}}$ each such triple $a,b$ and $c$ is counted 3 times, hence if we multiply the sum by $- \frac 13$ we get the common contribution of $ab, ac$ and $bc$ in $\sum_{ab \in E(G)} N_{ab} \cdot N_{ba} \cdot {_aN_{b}}$  is $-1$ and therefore the altogether contribution of $a,b$ and $c$ in $\sum_{ab \in E(G)} n_3(ab) - \frac 13 \sum_{ab \in E(G)} N_{ab} \cdot N_{ba} \cdot {_aN_{b}}$  is 2, hence:
$$SW_3(G) = \sum_{ab \in E(G)} n_3(ab) - \frac 13 \sum_{ab \in E(G)} N_{ab} \cdot N_{ba} \cdot {_aN_{b}}.$$
\end{proof}

Not that for $k \geq 4$ no simple formula for the edge decomposition of the Steiner $k$-Wiener index of block graph $G$ seems to exist. The reason is that there are more and more different cases which needed to be considered when more than one common neighbour of $a$ and $b$ is included in some Steiner tree for a set of $k$ vertices.

\section{Vertex decomposition of the Steiner $k$-Wiener index of block graphs}
\label{section:vertexdecomposition}

For a graph $G$ and $v \in V (G)$, let $G \setminus v$ denote a graph obtained by removing
$v$ from $G$. Note that $G\setminus v$ may consists of several components and that their number equals the degree of $v$.
\begin{theorem}\label{thm2}
Let $G$ be a connected  block graph on $n$ vertices with the set of cut vertices $V_{c}(G)$ and $k$ an integer such that $2 \leq k \leq n$. Then
$$SW_k(G) =  \sum_{v \in V_{c}(G)} N_k(G \setminus v)+(k-1){n\choose k}.$$
\end{theorem}
\begin{proof}
$N_k(G \setminus v)$ counts the number of times a cut vertex $v$ is an inner vertex of Steiner tree. Since each such vertex adds 1 to the Steiner distance of a set of $k$ vertices, Steiner distance between $k$ vertices is by $k -1$ greater than the number of inner vertices in the corresponding Steiner tree, adding $k-1$ for each set of $k$ vertices, we get the sum of Steiner distances between all $k$
sets of vertices, and the equality in formula holds.
\end{proof}

Special case of the Theorem \ref{thm2} on trees has been proved in \cite{kovse2016} by Kov\v{s}e, who also introduced the $k$-Steiner betweenness centrality $B_k(v)$ of a vertex $v\in V (G)$ as the sum of the fraction of all Steiner trees connecting $k$ terminal vertices, that include $v$ as its inner vertex, across all Steiner trees connecting the same set of $k$ terminal vertices:
\begin{equation*}
B_k(v)= \sum_{
\substack{
A\subseteq V(G)\setminus\{v\}\\
|A|=k}
}
\frac{\sigma_A(v)}{\sigma_A},
\end{equation*}
\noindent where $\sigma_{A}$ denotes the number of all Steiner trees between vertices of $A$ in a graph $G$ and $\sigma_{A}(v)$ denotes the number of all Steiner trees between vertices of $A$ in a graph $G$ that include $v$ as an inner vertex. 

As pendant vertices do not appear as inner vertices of any Steiner tree of a block graph $G$, it follows that $B_k(v)=0$ for any pendant vertex $v$ and $B_k(v) > 0$ for all other vertices, if $k < n$. In fact $B_k(v) = N_k(G \setminus v)$ and we can now restate Theorem \ref{thm2} as follows.

\begin{theorem}\label{sbetweenness}
Let $G$ be a connected  block graph on $n$ vertices with the set of cut vertices $V_{c}(G)$ and $k$ an integer such that $2 \leq k \leq n$. Then
\begin{equation*}
SW_k(G) = \sum_{v\in V_c(G)}B_k(v) + (k-1)\binom{n}{k}.
\end{equation*}
\end{theorem}

Note that in \cite{kovse2016} it has been shown that the equality $SW_k(G) = \sum_{v\in V(G)}B_k(v) + (k-1)\binom{n}{k}$ holds for any graph $G$.

\begin{corollary}\label{cor:universalvertex}
Let $G$ be a star-like block graph on $n$ vertices and blocks $B_1, B_2, \ldots , B_t$ of orders $b_1, b_2, \ldots, b_t$ and $k$ an integer such that $2 \leq k \leq n$. Then 
\begin{equation*}
SW_k(G)  =  (n-1)\binom{n-1}{k-1}- \sum_{i=1}^t\binom{b_i - 1}{k}.
\end{equation*}
\end{corollary}

\begin{proof}
Let $v$ be the universal vertex of $G$. Hence $v$ is the only cut vertex of $G$. Then $v$ is not an inner vertex of a Steiner tree connecting vertices of $S$ if and only if $S \subseteq B_i \setminus \{v\}$ for some $i \in \{1, \ldots, t\}$, hence $B_k(v)=\binom{n-1}{k} - \sum_{i=1}^t\binom{b_i - 1}{k}$ and by Theorem \label{sbetweenness}
\begin{align*}
SW_k(G)  &= B_k(v) + (k-1)\binom{n}{k}\\
&=  \binom{n-1}{k} - \sum_{i=1}^t\binom{b_i - 1}{k} + (k-1)\binom{n}{k}.
\end{align*}
From basic binomial identity $\binom{n}{k}= \binom{n-1}{k-1} + \binom{n-1}{k}$ it follows that
$\binom{n-1}{k}+ (k-1)\binom{n}{k} =  \binom{n}{k} - \binom{n-1}{k-1} + (k-1)\binom{n}{k}  = k \binom{n}{k} - \binom{n-1}{k-1} = n \binom{n-1}{k-1} - \binom{n-1}{k-1} = (n-1)\binom{n-1}{k-1}$, hence the formula follows.

\end{proof}

\begin{corollary}\label{cor:windmill}
Let $k$ be an integer such that $2 \leq k \leq n$. For the windmill graph $Wd(r,t)$,
\begin{equation*}
SW_k(Wd(r,t)) =  (n-1)\binom{n-1}{k-1}- t\binom{r-1}{k}.
\end{equation*}
\end{corollary}

Note that for $r=2$ the formula from Corollary \ref{cor:windmill} gives formula $SW_k(S_n) =  (n-1)\binom{n-1}{k-1}$ for a star $S_n$ on $n$ vertices, already noted in \cite{gutman2016discuss}.




\begin{corollary}\label{cor:clawfreeformula}
Let $G$ be a path-like block graph on $n$ vertices, with blocks $B_1, B_2, \ldots , B_t$ of orders $b_1, b_2, \ldots, b_t$ ordered from one pendant block to the other, and $k$ an integer such that $2 \leq k \leq n-1$. Then
 
$$SW_k(G) = (k-1)\binom{n}{k} + (t-1){n-1\choose k} -  \left( \sum_{i=1}^{t-1} \binom{(\sum_{j=1}^{i} b_j) - i}{k} + \sum_{i=1}^{t-1} \binom{n + i - \sum_{j=1}^{i} b_{j}}{k} \right).$$
\end{corollary}

\begin{proof}
Let $v_1, v_2, \ldots, v_{t-1}$ denote cut vertices of $G$, where $v_i \in B_i \cap B_{i+1}$ for $i \in \{1, \ldots, t-1 \}$. Then $v_i$ is not an inner vertex of a Steiner tree connecting vertices of $S$ if and only if $S \subseteq B_1 \cup \ldots \cup B_i \setminus \{v_i)\}$ or $S \subseteq B_{i+1} \cup \ldots \cup B_t \setminus \{v_i)\}$. Hence

\begin{align*}
B_k(v_i) &= \binom{n-1}{k} - \binom{(\sum_{j=1}^{i} b_j) - i}{k} - \binom{n + i - \sum_{j=1}^{i} b_{j}}{k},
\end{align*}
and the formula follows.
\end{proof}

Clearly $SW_{n}(G)=n-1$ for any graph of order $n$. For $k = n - 1$ we can extend the Theorem on trees from \cite{gutman2016discuss} as follows.
\begin{theorem}
Let $G$ be a block graph of order $n$ with $p$ pendant vertices. Then 
 $$SW_{n-1}(G)= n^2 - n-p.$$ 
\end{theorem}
\begin{proof}

If $k=n-1,$ then subset $S$ contains all but one vertex from $G.$ If the vertex missing from $S$ is pendant, then the vertices contained in $S$ form a connected subgraph and its spanning tree of order $n-1$ provides the Steiner tree for $S$. Therefore $d(S)=n-2.$ There are $p$ such subsets, contributing to $SW_{n-1}$ by $p(n-2).$

If the vertex not present in $S$ is a cut vertex of $G$ and the respective Steiner tree must contain all vertices of $G.$ Therefore a spanning tree of order $n$ provides the Steiner tree for $S$ and $d(S)=n-1.$ There are $n-p$ such subsets, contributing to $SW_{n-1}$ by $(n-p)(n-1)$. Thus $SW_{n-1}(G)= p(n-2)+(n-p)(n-1)$, which proves the theorem.
\end{proof}

\section{Extremal values of Steiner $k$-Wiener index and GBS poset of block graphs}
\label{section:extremes}

 Among all trees on $n$ vertices star has the minimal Wiener index and path has the largest Wiener index. More generally in extremal problems concerning trees on $n$ vertices it turns out that the maximal (minimal) value of the examined parameter is attained at the star and the minimal (maximal) value is attained at the path. While often it is not hard to prove the extremality of the star, it turns out that to prove the extremality of the path usually requires some effort. 
 
Let $\mathcal{G}(b_1, b_2, \ldots , b_t)$ denote
the set of all connected block graphs with block order sequence $b_1 \geq b_2 \geq  \ldots \geq b_t$.  We will demonstrate that similar situation, as in the case of the extremality of the Wiener index over trees, happens also for the problem of the extremal values of the Steiner $k$-Wiener over $\mathcal{G}(b_1, b_2, \ldots , b_t)$.  If $t \leq 2$ then $\mathcal{G}(b_1, b_2, \ldots , b_t)$ consist of a unique block graph, hence in what follows we always assume that $t \geq 3$.

\begin{theorem}\label{minmax}
Let $G \in \mathcal{G}(b_1, b_2, \ldots , b_t)$  and $2 \leq k < n- 1$. Then
\begin{equation}
\label{eq:bounds}
(n-1)\binom{n-1}{k-1}- \sum_{i=1}^t\binom{b_i - 1}{k}  \leq SW_k(G).  
\end{equation}
 The equality is attained for the star-like block graph.
\end{theorem}

\begin{proof} Note that $d(S) \geq k-1$ for any subset $S$ of $k$ vertices, with equality if and only if Steiner tree connecting $S$ has no inner vertices. Note that the only vertices that can appear as inner vertices of some Steiner trees in a block graph are cut vertices. Hence $SW_k(G)$ will attain the minimum if and only if $G$ has only one cut vertex, which is then a universal vertex of $G$. Therefore we can use Corollary \ref{cor:universalvertex} and the lower bound is proved.
\end{proof}

In \cite{csikvari2010poset} Csikv{\'a}ri introduced a graph transformation called generalized tree shift and defined a partially ordered set (poset) on the set of unlabelled trees with $n$ vertices. The minimal element of this level poset is the path, and the maximal element is the star. Among other results, he showed that going up on this poset decreases the Wiener index. 

Next we introduce a graph transformation called a generalized block shift which defines a poset on $\mathcal{G}(b_1, b_2, \ldots , b_t)$. We will show that going up on this poset does not increase the Steiner $k$-Wiener index, which will help us to describe block graphs with minimal and maximal Steiner $k$-Wiener index among all block graphs from $\mathcal{G}(b_1, b_2, \ldots , b_t)$.


Let $G_2 \in \mathcal{G}(b_1, b_2, \ldots , b_t)$ and $x,y \in V(G_2)$, such that all blocks having nonempty intersection with the path between $x$ and $y$ (if they exist) have exactly two cut vertices. The generalized block shift (GBS) of $G_2$ is the block graph $G_1$ obtained from $G_2$ as follows: let $z$ be the neighbor of $y$ lying on the path between $x$ and $y$, we erase all the edges
between $y$ and $N(y)\setminus \{z\}$ and add the edges between $x$ and $N(y)\setminus \{z\}$, see
Figure \ref{fig3}. Note that $G_1 \in \mathcal{G}(b_1, b_2, \ldots , b_t)$.

\begin{center}
\begin{figure}[h!]
\centering
  \includegraphics[width= 1.2\textwidth]{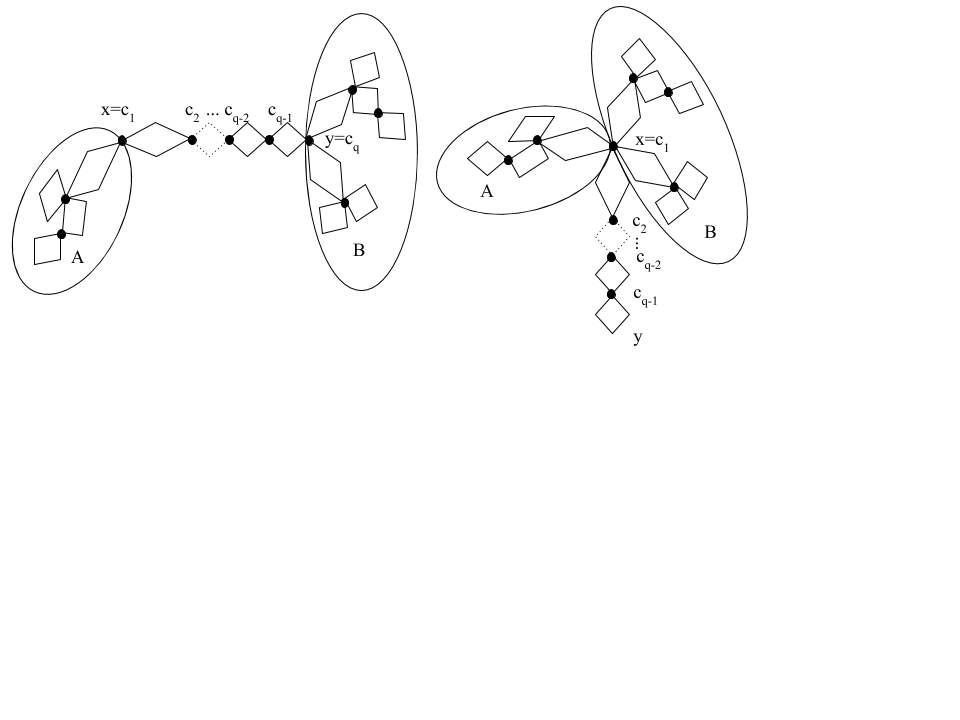} \vspace{-75mm}
  \caption{Proper generalized block shift transforming block graph $G_2$ on the left to the block graph $G_1$ on the right. Parallelograms denote blocks of $G_1$ and $G_2$, vertices $x=c_1,c_2, \ldots, c_q$ denote cut vertices.}
  \label{fig3}
\end{figure}
\end{center}
\vspace{-10mm}

We denote the cut vertices on the path between $x$ and $y$ of length $q$ by $c_1, c_2, \ldots ,c_q$, where $x=c_1$ and $y=c_q$. The set $A \subseteq V (G_2)$ consists of the vertices which can be reached with
a path from $c_q$ only through $c_1$, and similarly the set $B \subseteq V (G_2)$ consists of
those vertices which can be reached with a path from $c_1$ only through $c_q$. For
the sake of simplicity we denote the corresponding sets in $G_1$ also with $A$ and $B$. Furthermore let $A_0= N(x) \cap A$ in both $G_1$ and $G_2$, and let $B_0= N(x) \cap B$ in $G_1$ and $B_0= N(y) \cap B$ in $G_2$.

As in \cite{csikvari2010poset} we call $x$ the beneficiary and $y$ the candidate (for being a leaf) of the generalized block shift. Note that if $x$ or $y$ is a leaf in $G_2$, then $G_1$ and $G_2$ have the same number of pendant blocks, otherwise the number of pendant blocks in $G_1$ equals the number of pendant blocks in $G_2$ plus 1. In the latter case we call the generalized block shift proper.

Let $G_1, G_2 \in \mathcal{G}(b_1, b_2, \ldots , b_t)$. We denote by $G_1 > G_2$ if $G_1$ can be obtained from $G_2$ by some proper generalized block shift. The relation $>$ induces a poset on $\mathcal{G}(b_1, b_2, \ldots , b_t)$, which we call GBS poset.

We can always apply a proper generalized block shift to any block graph which has
at least one non pendant block. Therefore the only maximal
element of GBS poset is the (unique) star-like block graph from $\mathcal{G}(b_1, b_2, \ldots , b_t)$. The following theorem shows that
the minimal elements of the induced poset are path-like block graphs.

\begin{theorem}\label{thmGBS}
Every block graph from  $\mathcal{G}(b_1, b_2, \ldots , b_t)$ that is not a path-like block graph is the image of some proper generalized block shift.
\end{theorem}

\begin{proof}   
 Let $G$ be a block graph from $\mathcal{G}(b_1, b_2, \ldots , b_t)$ that is not path-like, i. e., it has at least one cut vertex belonging to at least 3 blocks. Let $v$ be a leaf of $G$ and $w$ the closest cut vertex to $v$ which belongs to at least 3 blocks. Then the cut vertices (if they exist) on the path between $v$ and $w$ belong to exactly 2 blocks. Vertex $w$ belongs to at least two blocks different from the one which
has a nonempty intersection with the path between $v$ and $w$, so we can split $N(w)$
into two nonempty sets $A_0$ and $B_0$, where vertices of $A_0$ consist of those that form a block which has nonempty intersection with the path between $v$ and $w$. Let $G'$ be the block graph obtained by erasing the
edges between $w$ and $B_0$ and adding the edges between $v$ and $B_0$. Then $G$ can be obtained from $G'$ by a generalized block shift, where $w$ is the beneficiary and $v$ is the
candidate. Since both $A_0$ and $B_0$ are nonempty this is a proper generalized block shift.
\end{proof}

\begin{corollary}\label{corGBS}
The star-like block graph is the unique maximal element of the GBS poset on $\mathcal{G}(b_1, b_2, \ldots , b_t)$. The set of minimal elements consist of all path-like block graphs.
\end{corollary}


\begin{theorem}
\label{thm:decrease}
A proper generalized block shift decreases the value of the Steiner $k$-Wiener index if and only if $|A \cup B| \geq k$, otherwise the value remains the same.
\end{theorem}

\begin{proof}
Let $G_2$ be a block graph and $G_1$ its image obtained by a proper generalized block shift, moreover let $ 2 \leq k \leq n-2$, where $n=|V(G_1)|=|V(G_2)|$. Let $d^1$ and $d^2$ denote the Steiner distance in $G_1$ and $G_2$ respectively.
For $i \in \{1,\ldots, q\}$ it follows that
$d^1(\{c_i\} \cup A_1) + d^1\left(\{c_{q + 1 - i}\} \cup A_1\right) = d^2(\{c_i\} \cup A_1) + d^2(\{c_{q + 1 - i} \cup A_1)$ for all $A_1 \subseteq A$ and 
$d^1(\{c_i\} \cup B_1) + d^1(\{c_{q+ 1 - i}\} \cup B_1) = d^2(\{c_i \} \cup B_1) + d^2(\{c_{q + 1 - i} \cup B_1)$ for all $B_1 \subseteq B$. Moreover
$d^1(A_1) = d^2(A_1) \text{ for } A_1 \subseteq A,$ and $d_1(B_1) = d_2(B_1) \text{ for } B_1 \subseteq B${ and} $d^2(A_1 \cup B_1) = d^1(A_1 \cup B_1) + (q - 1) \text{ for } A_1 \subseteq A  \text{ and } B_1 \subseteq B.$
Altogether we have 
\begin{align*}
SW_k(G_2) &= \sum_{\substack{S\subseteq V(G_2)\\ | S| =k}} d^2(S) =  \sum_{\substack{S\subseteq V(G_1)\\ | S| =k}} d^1(S) + (q-1)\cdot\hspace{-5mm}
\sum_{\substack{l_1+l_2=k\\
0< l_1,l_2<k}} \binom{|A|}{l_1}\binom{|B|}{l_2}  \\
&= SW_k(G_1) + (q-1)\cdot\hspace{-5mm}\sum_{\substack{l_1+l_2=k\\
0< l_1,l_2<k}} \binom{|A|}{l_1}\binom{|B|}{l_2}.
\end{align*}
If $|A \cup B| \geq k$ then the generalized block shift decreases Steiner $k$-Wiener index, otherwise the value remains the same.
\end{proof}

Hence going up the GBS poset never inscreases the Steiner $k$-Wiener index. Moreover from the Theorem \ref{thm:decrease} and the fact that the path-like block graphs are minimal elements, and the star-like block graph is the only maximal element of the poset induced by the generalized block shift on $\mathcal{G}(b_1, b_2, \ldots , b_t)$, we get the following corollary.

\begin{corollary}
\label{cor:minmax}
Let $2\leq k \leq n$.
Among all block graphs from $\mathcal{G}(b_1, b_2, \ldots , b_t)$ the minimum value of the Steiner $k$-Wiener index is attained by the star-like block graph, and the maximum value of the Steiner $k$-Wiener index is attained among path-like block graphs.
\end{corollary}
%

Hence Corollary \ref{cor:minmax} implies Theorem \ref{minmax}. Finding the upper bound of the Steiner $k$-Wiener index over $\mathcal{G}(b_1, b_2, \ldots , b_t)$ turns out to be more challenging problem. Let $\mathcal{G_{CF}}(b_1, b_2, \ldots , b_t)$ denote the set of all connected claw-free block graphs with block order sequence $(b_1, b_2, \ldots , b_t)$, where $b_1 \geq b_2 \geq  \ldots \geq b_t$. Recall that claw-free block graphs are precisely the line graphs of trees. In~\cite{Buckley81} Buckley has shown the following result.

\begin{theorem}
\label{thm.buckley}
Let $T$ be a tree on $n$ vertices. Then $W(L(T))=W(T) - {n \choose 2}$. 
\end{theorem}

Note that for any $G \in \mathcal{G_{CF}}(b_1, b_2, \ldots , b_t)$, any cut vertex of $G$ is incident to exactly two blocks, therefore $G$ has maximal possible number of cut vertices among all block graphs from $\mathcal{G}(b_1, b_2, \ldots , b_t)$. Hence all graphs from $\mathcal{G_{CF}}(b_1, b_2, \ldots , b_t)$ have the same number of cut vertices, which is equal to $t-1$, and therefore also the same number of pendant vertices, which equals $n-t+1$.

By Theorem \ref{thm.buckley} finding minimum or maximum value of the Wiener index over the set $\mathcal{G_{CF}}(b_1, b_2, \ldots , b_t)$ is equivalent to  finding a minimum or maximum value of the Wiener index on the set of trees with the degree sequence $(b_1, \ldots, b_t, 1, 1, \ldots, 1)$, where 1's corresponds to pendant vertices in block graphs from $\mathcal{G_{CF}}(b_1, b_2, \ldots , b_t)$, hence 1 appears $n-t+1$ times in the sequence.

Let $(d_1, \ldots, d_n)$ be a degree sequence of the tree $T$ with
$d_1 \geq d_2 \geq  \ldots \geq  d_k \geq  2 > 1 = d_{k+1} = \ldots = d_n$.
Then $T$ is called greedy tree if it can be embedded in the plane as follows:

\begin{enumerate}
\item take the vertex $v$ with degree $d_1$ as the root,
\item each vertex $u$ lies on some line $i$ where $i$ is the distance between the root $v$ and $u$,
\item each line is filled up with vertices in decreasing degree order from left to right.
\end{enumerate}
Next we provide an example of a greedy tree and its line graph.
\begin{ill}
Let $(4, 4, 4, 3, 3, 3, 2, 2, 2, 1, 1, 1, 1, 1, 1, 1, 1, 1, 1, 1)$ be the given degree sequence. Then its corresponding greedy tree using the plane embedding from the definition is
shown in Figure \ref{fig_greedy}.
\end{ill}

\begin{center}
\begin{figure}[h!]
\centering
\includegraphics[width= 0.9\textwidth]{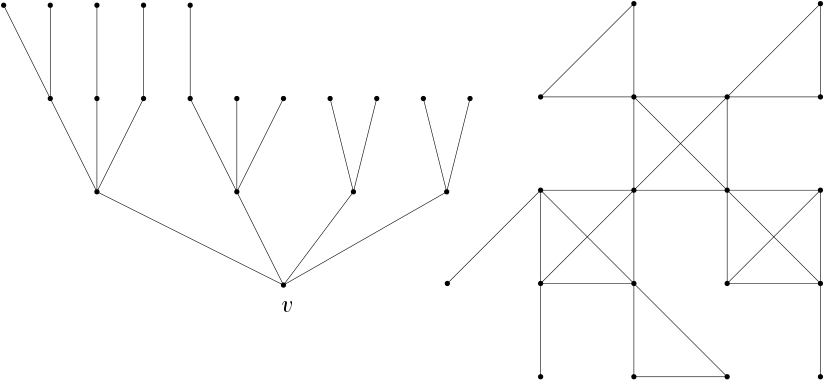}
  \caption{A greedy tree $T$ and its line graph $L(T)$.}
  \label{fig_greedy}
\end{figure}
\end{center}

Wang \cite{wang2008extremal} and Zhang et al. \cite{zhang2008wiener} have shown
independently the following theorem.

\begin{theorem}
\label{thm.greedytree}
Let $(d_1, \ldots, d_n)$ be an integer sequence with
$d_1 \geq d_2 \geq \ldots \geq d_s \geq 2 > 1 = d_{s+1} = \ldots= d_n$
and $\sum_{i=1}^n d_i = 2(n-1)$.
Then the greedy tree with degree sequence $(d_1, \ldots, d_n)$ minimizes the Wiener index among
all trees with same degree sequence.
\end{theorem}

From Theorem \ref{thm.buckley} and Theorem \ref{thm.greedytree} we immediately get the following corollary.

\begin{corollary}\label{cor.greedy}
Then line graph of the greedy tree with the degree sequence\\ $(b_t, b_{t-1}, \ldots, b_2, b_1, 1, 1, \ldots, 1)$ minimizes the Wiener index over the set $\mathcal{G_{CF}}(b_1, b_2, \ldots , b_t)$.
\end{corollary}

The maximization problem is somewhat more complicated. It has been proven by Shi in \cite{shi1993average} that the problem can be reduced to the study of caterpillars. Schmuck et al. \cite{schmuck2012greedy} have shown that an
optimal tree is a caterpillar with vertex degrees non-increasing from the ends of
the caterpillar towards its central part. 
In \cite{ccela2011wiener} Schmuck et al. obtained a polynomial time algorithm for finding a caterpillar $T$ that maximizes the Wiener index among all trees with a prescribed degree sequence $(b_t, b_{t-1}, \ldots, b_2, b_1, 1, 1, \ldots, 1)$. We call $T$ the optimal caterpillar for the degree sequence $(b_t, b_{t-1}, \ldots, b_2, b_1, 1, 1, \ldots, 1)$.

\begin{corollary}
\label{cor.caterpillar}
The line graph of the optimal caterpillar for the degree sequence\\ $(b_t, b_{t-1}, \ldots, b_2, b_1, 1, 1, \ldots, 1)$ maximizes the Wiener index over the set $\mathcal{G_{CF}}(b_1, b_2, \ldots , b_t)$.
\end{corollary}
 
As the geodetic distance and the Steiner distance on trees share many common properties, 
we propose the following problem.

\begin{problem}\label{conjecture1}
Let $k$ be an integer with $3\leq k \leq n$.
Let $(d_1, \ldots, d_n)$ be an integer sequence with
$d_1 \geq d_2 \geq \ldots \geq d_s +\geq 2 > 1 = d_{s+1} = \ldots= d_n$
and $\sum_{i=1}^n d_i = 2(n-1)$.
Is it true that the greedy tree minimizes the Steiner $k$-Wiener index over the set of all trees with degree sequence $(d_1, \ldots, d_n)$, and that the optimal caterpillar maximizes the Steiner $k$-Wiener index over the set of all trees with degree sequence $(d_1, \ldots, d_n)$?
\end{problem}

Although we are not aware of a possible generalization of Buckley's theorem which would relate the Steiner $k$-Wiener index of a tree to the Steiner $k$-Wiener index of its line graph, due to the fact that Steiner distance behaves nicely on block graphs, as demonstrated throughout this paper, we propose also the following problem.

\begin{problem}\label{conjecture2}
Let $k$ be an integer with $3\leq k \leq n$.
Let $(d_1, \ldots, d_n)$ be an integer sequence with
$d_1 \geq d_2 \geq \ldots \geq d_s \geq 2 > 1 = d_{s+1} = \ldots= d_n$
and $\sum_{i=1}^n d_i = 2(n-1)$. Is it true that the line graph of the greedy tree with the degree sequence $(d_1, \ldots, d_n)$ minimizes the Steiner $k$-Wiener index over the set ${\mathcal G}(d_1, \ldots, d_n)$, and that the line graph of the optimal caterpillar for the degree sequence $(d_1, \ldots, d_n)$ maximizes the Steiner $k$-Wiener index over the set ${\mathcal G}(d_1, \ldots, d_n)$?
\end{problem}

Positive solution of Problem \ref{conjecture1} seems to be the crucial step towards the positive solution of Problem \ref{conjecture2}.

 %

\frenchspacing                                

\vspace{-2mm}
\Addresses
                    
\end{document}